\documentclass{tac}
\usepackage{enumerate,xspace}
\usepackage{amsmath,xspace,amssymb}
\usepackage[dvips]{graphics}
\usepackage{proof}
\usepackage{latexsym}  
\usepackage{euscript}  
\usepackage[svgnames]{xcolor}
\usepackage[normalem]{ulem}

\newcommand{\del}[1]{}
\newcommand{\add}[1]{#1}
\newcommand{\ed}[1]{}

\newtheorem{observation}{Remark}[section]
\newtheorem{lemma}[observation]{Lemma}  
\newtheorem{theorem}[observation]{Theorem}
\newtheorem{definition}[observation]{Definition}
\newtheoremrm{example}[observation]{Example}

\newtheorem{proposition}[observation]{Proposition} 
\newtheorem{corollary}[observation]{Corollary}

\newtheoremrm{untitled}[observation]{} 
\newtheoremrm{remark}[observation]{Remark}

\numberwithin{equation}{section}

\newcommand*{\emptybox}{\leavevmode\hbox{}}

\newcommand{\<}{\langle}
\renewcommand{\>}{\rangle}

\newcommand{\A}{\ensuremath{\mathbb A}\xspace}
\newcommand{\B}{\ensuremath{\mathbb B}\xspace}
\newcommand{\C}{\ensuremath{\mathbb C}\xspace}
\newcommand{\D}{\ensuremath{\mathbb D}\xspace}

\newcommand{\J}{\ensuremath{\mathbb J}\xspace}
\newcommand{\N}{\ensuremath{\mathbb N}\xspace}

\newcommand{\p}{\pi}
\newcommand{\zq}{0_{\sf q}}
\newcommand{\pq}{+_{\sf q}}

\newcommand{\rarr}{\rightarrow}
\newcommand{\Aff}{\ensuremath{\operatorname{\textnormal{\textsf{Aff}}}}}
\newcommand{\AAff}{\ensuremath{\operatorname{\textnormal{$\mathbb{A}$\textsf{ff}}}}}
\newcommand{\Geom}{\mbox{\sf Geom}}
\newcommand{\GeomTf}{\Geom_{\sf tf}}
\newcommand{\GeomFlat}{\Geom_{\sf flat}}
\newcommand{\Mf}{\ensuremath{{\sf Mf}}}
\newcommand{\Tan}{\ensuremath{\operatorname{\textnormal{Tan}}}}
\newcommand{\Cat}{\mbox{\textnormal{Cat}}}
\newcommand{\ob}{\ensuremath{\operatorname{\textnormal{\textsf{ob}}}}}

\newcommand{\sF}{\ensuremath{\mathcal{F}}}

\input xy
\xyoption{all}
\xyoption{2cell}
\UseAllTwocells

\input{diagxy}

\title{Affine geometric spaces in tangent categories}

\author{R. F. Blute, G. S. H. Cruttwell,  and R. B. B. Lucyshyn-Wright}

\address{Department of Mathematics and Statistics, University of Ottawa, Ottawa, Ontario CANADA \\ Department of Mathematics and Computer Science, 
Mount Allison University, Sackville, New Brunswick CANADA \\ Department of Mathematics and Computer Science, Brandon University, Brandon, Manitoba CANADA}

\eaddress{rblute@uottawa.ca; gcruttwell@mta.ca; lucyshyn-wrightr@brandonu.ca}

\thanks{The authors thank NSERC for its generous support.  The third author gratefully acknowledges an AARMS PDF held earlier in the development of this work.}

\keywords{Tangent categories, affine manifolds, connections}
\amsclass{18D99, 53A15, 53B05, 53C05, 18F15} 
\copyrightyear{2019}

\begin{document}
\maketitle

\begin{abstract} We continue the program of {\it structural differential geometry} \add{that} begins with the notion 
of a tangent category, an axiomatization of \add{structural aspects of }the tangent functor on the category of smooth manifolds.    In classical geometry, having an affine structure on a manifold is equivalent to 
having a flat torsion-free connection on its tangent bundle. This equivalence allows us to define a category of affine \add{objects} associated to a tangent category and we show that the resulting category is also a tangent category, as 
are several related categories.   As a consequence of some of these ideas we also give two new characterizations of flat torsion-free connections.  

We also consider 2-categorical structure associated to the category of tangent categories and demonstrate that 
assignment of the tangent category of affine objects to a tangent category induces a 2-comonad.

\end{abstract}


\section{Introduction}

This paper is part of a broader program of {\it structural differential geometry.}  The idea is to axiomatize \add{structural aspects of }the 
category of \add{smooth} manifolds and smooth maps by axiomatizing the tangent functor. A category
with an abstract tangent functor is called a {\it tangent category}. An axiomatization of the tangent functor 
was first given by Rosick\'y \cite{Ros}, and the concept was elaborated on
by Cockett and the second author in \cite{CC1} with additional ideas introduced 
in \cite{CC2,CC3,CC4,CL}. 

The axioms of a tangent category are sufficiently strong that one can develop highly nontrivial results,
but also general enough to capture a number of different settings where there is a sensible notion of 
smoothness.  For example, models of synthetic differential geometry fit into this framework, as do the convenient manifolds of Fr\"olicher, Kriegl and Michor \cite{FK,KM,BET}, and a notion of differentiation appearing in the calculus of functors \cite{abFuncCalcDirDeriv}.  

There is also a strong logical underpinning for tangent categories. {\it Differential linear logic} \cite{ER1,ER2}
and the associated categorical structures of {\it differential categories} \cite{BCS1} \add{are} an extension
of linear logic to include an inference rule capturing the operation of taking a directional derivative. To every
differential category, one can associate a coKleisli category. Such coKleisli categories are examples of 
{\it cartesian differential categories} \cite{BCS2}. Every cartesian differential category has a canonical tangent category
structure \cite[Section 4.2]{CC1}.   

Following this initial work, there has been further development showing the extent to which \add{further ideas of differential geometry can be developed within the abstract setting of tangent categories}. See \cite{CC2,CC3,CC4,CL}. This paper is 
another contribution to this program.   In particular, we wish to look at affine manifolds (manifolds with an atlas whose transition maps are affine) and how they can be defined in tangent categories.   Affine differential 
geometry is a relatively small subindustry of differential geometry, but affine manifolds have a great deal of
interesting structure as well as a variety of examples. See \cite{NS,Goldm1,Goldm2,Au,AuMar}.

The work of this paper is in part inspired by the thesis of Jubin \cite{Jub}. Jubin considers the tangent functor on the 
category of smooth manifolds and smooth maps. He demonstrates that this tangent functor has a monad structure.
Indeed he shows \add{that }it has precisely one such and further demonstrates that there are no comonad structures
on the tangent functor at all. However he does demonstrate that when one restricts to the subcategory 
of affine manifolds and affine maps, there are infinite families of monads and comonads. Furthermore
there are mixed distributive laws \add{(see, e.g., }\cite{MW}\add{)} between these structures. We intend to present 
an abstraction of Jubin's systems of monads and comonads in a sequel. 

While the notion of system of affine charts is not directly amenable to definition in a tangent category, we use a theorem of 
Auslander and Markus \cite{AuMar} showing that an affine manifold can be defined equivalently as a manifold equipped with a flat 
torsion-free connection on its tangent bundle (Theorem \ref{thm:aus-mar} below).  Thus we are led to the theory of connections in 
a tangent category as introduced in \cite{CC4}. See \cite{DP} for the classical theory of connections. Perhaps the
best evidence of the strength of the axioms for \add{tangent categories} is how well the theory of connections
works here.  Prompted by the Auslander-Markus characterization of affine manifolds, we define a {\it geometric space} 
to be an object equipped with a connection\add{ on its tangent bundle} and an {\it affine geometric space} to be a geometric space 
\add{whose} associated connection is flat and torsion-free. It is reasonable to call such objects geometric since the \add{given} 
connection \add{generalizes Riemannian structure and} allows one to define such geometric \add{features} as 
curvature and torsion. \add{In particular, such features can be defined for an 
object with connection in an arbitrary tangent category}. \add{Maps in the category of geometric spaces are those maps that 
commute with the given connections.}  There has not been much in the way of study of categories with such morphisms,
although we do mention \cite{Hicks}. 

We first show that the various geometric categories that we define remain tangent categories, with the structure lifting 
from the base category. Along the way, we also look at \add{certain} notions of morphism between tangent categories 
and derive technical lemmas about their lifting to the geometric categories. We also give an alternative 
characterization of flat torsion-free connections that seems to be new (Theorem \ref{propFTFConditions}): a flat torsion-free connection \add{$K$ }can be seen as a morphism in the category of geometric spaces from $T(TM)$\add{ to }$TM$\add{, where the tangent bundles $TM$ and $T(TM)$ are endowed with geometric structures canonically induced by $K$}. This result alone demonstrates the importance of considering categories whose arrows commute with connections. 

We also consider \add{certain} 2-categories of tangent categories. We show that \add{there are 2-functors that send each tangent category to its tangent category of geometric spaces or affine geometric spaces}. Of course, 
these results require a careful presentation of the 2-categorical structure of tangent categories. We show that the 
affine construction induces a 2-comonad on the 2-category of tangent categories. We then define an {\it affine 
tangent category} to be an Eilenberg-Moore coalgebra with respect to this comonad and give an alternate characterization
of these structures.

\begin{remark}
\emptybox
\begin{itemize}
\item Note that following previous work on tangent and differential categories \cite{BCS1, BCS2, CC1, CC2, CC3, CC4} we write our compositions in diagrammatic order\add{ unless otherwise indicated.  However, we write the application of a functor $F$ to a morphism $f$ as $F(f)$, as in the cited works, or as $Ff$; correspondingly, we write the composite of functors $F:\A \rightarrow \B$ and $G:\B \rightarrow \C$ in non-diagrammatic order as $GF$.}
\item In many of the longer calculations, we omit the subscripts on natural transformations to save space.  
\end{itemize}
\end{remark}

\section{Affine manifolds}\label{Aff}

We give an overview of the classical theory of affine manifolds and their associated connections. See 
\cite{Au,AuMar,NS}.\add{  An \textit{affine manifold} is a real manifold whose transition maps are affine (and hence necessarily smooth):}

\begin{definition}\label{def:aff_mf} \add{An $n$-dimensional \textbf{affine manifold} is a real manifold equipped with a specified atlas consisting of charts
$\psi_i:U_i\xrightarrow{\sim} V_i \subseteq \mathbb{R}^n$ such that all of the composites \linebreak $\psi_i\circ\psi_j^{-1}\colon V_{ji} \rightarrow V_{ij}$ are affine maps between the subsets $V_{ji} = \psi_j(U_i\cap U_j)$ and $V_{ij} = \psi_i(U_i\cap U_j)$ of $\mathbb{R}^n$.  Here a map $f:V \rightarrow W$ between subsets $V \subseteq \mathbb{R}^n$ and $W \subseteq \mathbb{R}^m$ is said to be \textnormal{affine} if it has constant Jacobian, or equivalently, if it is the restriction of a map $F:\mathbb{R}^n \rightarrow \mathbb{R}^m$ that is affine in the usual sense, i.e. a composite of a linear map followed by a translation.}
\end{definition}

\add{While it is not obvious how one might generalize this notion to the more abstract context of tangent categories, there is a theorem of Auslander and Markus that will enable us to obtain such a generalization.}  See \cite{AuMar}. See \cite{DP} for the classical theory of connections.

\begin{theorem}[Auslander-Markus]\label{thm:aus-mar}
A manifold is an affine manifold if and only if there is a flat torsion-free connection on its tangent bundle. \add{  Moreover, each affine manifold has a canonically associated flat torsion-free connection on its tangent bundle, and every smooth manifold with a flat torsion-free connection on its tangent bundle carries an associated affine structure.}
\end{theorem}

Of special interest are the {\it complete affine manifolds}. These are affine manifolds that satisfy geodesic 
completeness, \cite{DP} p. 250. In the affine case, geodesic completeness turns out to be equivalent 
to being a quotient of an affine  space by a discrete group of affine transformations acting on the space. 
See \cite{Goldm2} for further discussion.

Affine manifolds are the objects of a category $\Aff$ in which a morphism is a \textbf{locally affine} map, i.e. a smooth map $f:M \rightarrow N$ such that for every pair of designated charts $U \cong V \subseteq \mathbb{R}^n$ and $U' \cong V' \subseteq \mathbb{R}^m$ for $M$ and $N$, respectively, the restriction of $f$ to $U \cap f^{-1}(U')$ has all its second partial derivatives equal to zero.

The category of affine manifolds, and certain structures that it carries, were studied by Jubin \cite{Jub}, using the following result:

\begin{theorem}
\add{The tangent functor on smooth manifolds lifts to an endofunctor on the category of affine manifolds $\Aff$.}
\end{theorem}

Jubin gives a proof of this result (of course, it also follows from our more general Theorem \ref{thm:aff-tcat}).  Jubin also studies monad and comonad structures on $\Aff$, as we discuss in \S \ref{sec:jub}.

\section{Tangent categories and connections}\label{sec:tcat_cx}



We assume the reader is familiar with tangent categories (the references \cite{CC1,CC3,CC4} are all suitable).  
While the axioms for a tangent category at first may appear ad-hoc, recent work of Leung \cite{leungClassifyingTanStructures} and Garner \cite{garnerEmbedding} has shown how tangent categories are related to \add{Weil} algebras and how tangent categories are a type of enriched category.





\smallskip

\add{Before defining a notion of connection in tangent categories, it is helpful to have at hand the following generalization of the notion of vector bundle, given in \cite{CC3}:}

\begin{definition}\label{def:db}
A \textbf{differential bundle} in a tangent category consists of an additive bundle $(q: E \to M, \pq: E_2 \to E, \zq: M \to E)$ 
with a map $\lambda: E \to T(E)$, called the {\bf lift}, such that
\begin{itemize}
        \item finite fibre powers of $q$ exist and are preserved by each $T^n$;
        \item $(\lambda,0_M)$ is an additive bundle morphism from $(E, q,\pq,\zq)$ to $(T(E), T(q), T(\pq), T(\zq))$;
	\item $(\lambda,\zq)$ is an additive bundle morphism from $(E, q,\pq,\zq)$ to $(T(E), p_E,+_E,0_E)$;
	\item the {\bf universality of the lift} requires that the following \add{be} a pullback:
	       $$\xymatrix{E_2 \ar[d]_{\p_0q=\pi_1q} \ar[rrr]^{\mu := \<\p_0\lambda, \p_10\>T(\pq)} & & & 
	       T(E) \ar[d]^{T(q)} \\ M \ar[rrr]_{0} && & T(M)}$$
              where $E_2$ is the pullback of $q$ along itself;
	\item the equation $\lambda \ell_E = \lambda T(\lambda)$ holds.
\end{itemize}
We shall write $\sf q$ to denote the entire bundle structure $(q,\pq,\zq,\lambda)$.

Now let $\sf q$ and $\sf q'$ be differential bundles.  A \textbf{bundle morphism} between these bundles 
simply consists of a pair of maps $f_1: E \to E'$, $f_0: M \to M'$ such that $f_1q' = qf_0$ 
(first diagram below).  A bundle morphism is \textbf{linear} in case, in addition, it preserves the lift, that 
is $f_1\lambda' = \lambda T(f_1)$ (the 
second diagram below).
$$\xymatrix{E \ar[d]_q \ar[rr]^{f_1} && E' \ar[d]^{q'} \\ M \ar[rr]_{f_0} && M'} ~~~~~~ 
 \xymatrix{E \ar[d]_\lambda \ar[rr]^{f_1} & & E' \ar[d]^{\lambda'} \\ T(E) \ar[rr]_{T(f_1)} && T(E')}$$
\end{definition}

\noindent Notably, every linear bundle morphism is automatically additive \cite[Proposition 2.16]{CC3}.

\add{In the present section, we shall tacitly assume that given differential bundles $\sf q$ satisfy the following additional condition, which is a prerequisite for considering connections on such bundles} \cite[Def. 2.2]{CC4}, \cite[3.1]{Lu:Cx}:
\begin{equation}\label{eq:assn-on-dbs}\begin{minipage}{5.7in}\textit{\add{For all natural numbers $n$ and $m$, the $n$-th fibre power $E_n \rightarrow M$ of $q$ has a pullback along the projection $T_m M \rightarrow M$, and this pullback is preserved by each $T^k$.}}\end{minipage}\end{equation}
Of course, the crucial example of a differential bundle is the tangent bundle
$${\sf p}_M = (p_M,+_M,0_M,\ell_M)$$
of an object $M$ \cite[Example 2.4]{CC3}, and we recommend keeping it in mind in the definitions below.  The tangent bundle always satisfies the preceding additional assumption \cite[Example 2.3]{CC4}.

\smallskip

Given any differential bundle ${\sf q} = (q:E \rightarrow M,\pq,\zq,\lambda)$ we obtain an associated differential bundle
$$T({\sf q}) = (T(q):TE \rightarrow TM, T(\pq), T(\zq), T(\lambda)c_E),$$
defined in \cite[\S 2.3]{CC3}.  Hence $TE$ underlies two differential bundles, namely $T({\sf q})$ and ${\sf p}_E$, whose underlying additive bundles appear in Definition \ref{def:db}.

\smallskip

We now recall the notion of connection as \add{it is }defined \add{in \cite{CC4} }with respect to a tangent category.\add{  In this formulation, a connection consists of two parts, a \textit{vertical connection} and a \textit{horizontal connection}, that are suitably compatible.  Together, they split the tangent bundle of a given bundle into vertical and horizontal components.}  A result of Patterson \cite[Theorem 1]{Patt}  shows that vertical connections in the category of smooth manifolds correspond to one of the standard formulations of the notion of connection:  a covariant (or Koszul) derivative.  On the other hand, horizontal connections in smooth manifolds correspond to what are known as linear Ehresmann connections \cite[Definition 7.2.1]{finsler}\footnote{Some authors such as Lang simply call these connections \cite[pg. 104]{Lang}.}.  For smooth manifolds the existence of a covariant derivative/vertical connection is equivalent to the existence of a linear Ehresmann/horizontal connection \cite[Proposition 7.5.11]{finsler}  but \add{this} is no longer the case in a general tangent category. \add{This led the authors of \cite{CC4} to employ} both\add{ notions at once,} as seen below.

A covariant derivative (or Koszul connection) is an operation on global sections of certain bundles.  Although this notion is one of the most standard formulations of connections, it is inadequate for describing connections in a general tangent category as these are structures internal to the given category and so cannot in general be characterized in terms of global sections.  However one can find an appropriate abstract definition based on the work of
Patterson \cite{Patt}.

\begin{definition}\label{defnVertConnection}
Let $\sf q$ be a differential bundle on $E$ over $M$.  \add{A \textbf{vertical connection} on $\sf q$ is a map $K:T(E) \rightarrow E$ that is a retraction of $\lambda:E \rightarrow T(E)$ and satisfies the following conditions:}
\begin{enumerate}[{\bf [C.1]}]
	\item $(K,p):T({\sf q}) \to {\sf q}$ is a linear bundle morphism;
	\item $(K,q): {\sf p}_E \to {\sf q}$ is a linear bundle morphism.  
\end{enumerate}
\end{definition}

Curvature in differential geometry is thought of as a measure of the extent to which a geometric\add{ space} deviates from being
flat $n$-space \cite{DP}. It is typically defined for Riemannian manifolds or\add{,} more generally\add{,} arbitrary 
manifolds equipped with a connection.  Of particular interest are those connections that have no curvature, that is, they are \emph{flat}.

\begin{definition}\label{defn:flat}
In a tangent category with a vertical connection $K$ on a differential bundle $\sf q$, say that the vertical connection is \textbf{flat} if $cT(K)K = T(K)K$.  
\end{definition}
In the tangent category of smooth manifolds, this is equivalent to the usual definition, by a result of Patterson \cite[Theorem 2]{Patt}.

The closely related notion of torsion \add{captures twisting effects that tangent vectors incur when} subject to parallel transport.  Again, it is of particular interest to see when connections have no torsion; that is, when they are \emph{torsion-free}.  

\begin{definition}\label{defn:torsionFree}
In a tangent category with a vertical connection $K$ on the tangent bundle of an object $M$ (so that $K: T^2(M) \to T(M)$), say that the vertical connection is \textbf{torsion-free} if $cK = K$.  
\end{definition}
Again, the equivalence of this definition with the standard one follows from a result of Patterson \cite[Theorem 3]{Patt}.

Linear Ehresmann connections also generalize to the setting of an arbitrary tangent category, through the notion of {\it horizontal connection} \cite{CC4}:

\begin{definition}\label{defnHorizConnection}
Let $\sf q$ be a differential bundle.  \add{A \textbf{horizontal connection} on $\sf q$ is a map $H: T(M) \times_M E \to T(E)$ 
that is a section of $U = \<T(q),p\>$ and satisfies the following conditions:}
\begin{itemize}
	\item $(H,1_E)$ is a linear bundle morphism from $q^*({\sf p_M})$ to $\sf p_E$;
	\item $(H,1_{T(M)})$ is a linear bundle morphism from $p^*({\sf q})$ to $T({\sf q})$.
\end{itemize}
\end{definition}
\noindent\add{Here, we write $q^*({\sf p_M})$ (resp. $p^*({\sf q})$) to denote the pullback of ${\sf p_M}$ along $q$ (resp. of ${\sf q}$ along $p_M$) }\cite[Lemma 2.7]{CC3}\add{, which exists as a consequence of our assumption \eqref{eq:assn-on-dbs}; see }\cite[2.4.7, 2.4.8, \S 3]{Lu:Cx}\add{ for an explicit account of this existence.}

As already noted, \add{for smooth manifolds} the notions of horizontal and vertical connection are equivalent.  \add{In general tangent categories} they are not, and this led Cockett and Cruttwell to define a connection to be a pair consisting of one of each satisfying compatibility, as follows.

\begin{definition}[{\cite[Def. 5.1]{CC4}}]\label{defnConnection}
A \textbf{connection}, $(K,H)$, on a differential bundle $\sf q$ consists of a vertical connection 
$K$ on $\sf q$ and a horizontal connection $H$ on $\sf q$ such that
\begin{itemize}
	\item $H K = \pi_1q\zq$;
	\item $\<K,p\>\mu + UH = 1_{T(E)}$
where \add{$\mu$ is as defined in \ref{def:db} and }the addition operation $+$ is induced by $+_E:T_2E \rightarrow TE$.
\end{itemize}
These two conditions are called the \textnormal{compatibility conditions} between $H$ and $K$.
 
\end{definition}

\begin{proposition}[{\cite[Prop. 3.5]{CC4}}]\label{diffObjectConnection}
\add{In a Cartesian tangent category, }\add{a}ny differential object $A$ (i.e. a differential bundle over $1$) has a canonical connection 
$(K,H)$ where $H = \pi_10$ \add{and $K$ is the \textnormal{principal projection} $\hat{p}:TA \rightarrow A$ associated to $A$ (}\cite[\S 3]{CC3}, \cite[Example 2.3]{CC4}\add{).}
\end{proposition}

\add{The third author proved that} connections \add{in a tangent category }are equivalently described as vertical connections satisfying a certain `exactness' condition, as follows\add{:}

\begin{theorem}[{\cite[8.2(3)]{Lu:Cx}}]\label{thm:evcx}
Let ${\sf q} = (q: E \to M, \pq, \zq,\lambda)$ be a differential bundle.  Then a connection on ${\sf q}$ is equivalently given by a vertical connection $K:TE \rightarrow E$ such that the following is a fibre product diagram in $\C$:
$$
\xymatrix{
            & TE \ar[dl]_{T(q)} \ar[d]|{p_E} \ar[dr]^K & \\
TM \ar[dr]_{\add{p_M}} & E \ar[d]|{\add{q}}                          & E \ar[dl]^q\\
   & M                                                 &
}
$$
\end{theorem}

\begin{corollary}[{\cite[8.4(3)]{Lu:Cx}}]\label{thm:acx}
Let $M$ be an object of a tangent category $(\C,T)$.  Then a connection on the tangent bundle of $M$ is equivalently given by a vertical connection $K:T^2M \rightarrow TM$ \add{on ${\sf p}_M$ }such that the following is a fibre product diagram in $\C$:
\begin{equation}\label{eq:acx-fp}
\xymatrix{
                 & T^2M \ar[dl]_{T(p_M)} \ar[d]|{p_{TM}} \ar[dr]^K & \\
TM \ar[dr]_{p_M} & TM \ar[d]|{p_M}                                 & TM \ar[dl]^{p_M}\\
                 & M                                              &
}
\end{equation}
\end{corollary}

Given a morphism $K$ as \ref{thm:acx}, the associated horizontal connection $H$ is characterized by the following:

\begin{theorem}[{\cite[7.8]{Lu:Cx}}]\label{thm:hcx-for-evcx}
Given a vertical connection $K:T^2M \rightarrow TM$ on ${\sf p}_M$ such that \eqref{eq:acx-fp} is a fibre product diagram, there is a unique horizontal connection $H$ such that $(K,H)$ is a connection on the tangent bundle of $M$ in the sense of \ref{defnConnection}.  Further, $H$ is the unique morphism $H:TM \times_M TM \rightarrow T^2M$ such that
$$HT(p_M) = \pi_0,\;\;\;\;Hp_{TM} = \pi_1,\;\;\;\;HK = p_20_M$$
where $p_2:TM \times_M TM \rightarrow M$ is the projection.
\end{theorem}

In the present paper, we shall represent connections as morphisms $K$ as in \ref{thm:evcx} and \ref{thm:acx}, but we shall also make important use of the associated horizontal connection $H$.  Moreover, throughout the rest of the paper, we shall primarily be concerned with connections on the tangent bundle of an object $M$.  Thus, for brevity, rather than speak of a ``connection on the tangent bundle of $M$'', we shall simply say ``connection on $M$''.  (Connections on tangent bundles are typically referred to as \emph{affine connections}, but, in this paper, this would cause an overload of the term ``affine''.)

\begin{definition}\label{def:acx-v}
In view of Corollary \ref{thm:acx} and the discussion above, we will call a morphism $K:T^2M \rightarrow TM$ a \textbf{connection on $M$} if $K$ is a vertical connection on the tangent bundle of $M$ and makes \eqref{eq:acx-fp} a fibre product diagram.
\end{definition}

We now record some properties of this notion of connection.

\begin{proposition}\label{prop:propK}
If $K$ is a connection on $M$, then 
\begin{enumerate}[(a)] 
\item ($K$ is a retract of $\ell$) $\ell_M K= 1_{TM}$;
\item ($K$ is a bundle morphism) $K p_M = p_{TM} p_M= T(p_M) p_M$;
\item (linearity of $K$) $K \ell_M = \ell_{TM} T(K)$ and $K \ell_M = T(\ell_M)c_{TM} T(K)$;
\item (additivity of $K$) $0_{TM} K = p_M0_M$, $+_{TM}K = \<\pi_0K, \pi_1 K\>+_M$, and $T(0_M)K = p_M0_M$, $T(+)K = \<T(\pi_0)K, T(\pi_1)K\>+_M$.  
\end{enumerate}
\end{proposition}
\begin{proof}
Properties (a)-(c) follow directly from the definition of a vertical connection on the tangent bundle (see \cite[Lemma 3.3]{CC4}, while (d) follows since linear bundle morphisms are additive, as noted above.
\end{proof}

\section{Geometric and affine structures in tangent categories} 

A manifold may carry further geometric structure, such as Riemannian structure, and it is only with reference to such additional structure that one can define several important aspects of its geometry, including geodesics, curvature, and parallel transport.  The notion of connection captures such structure by means of a formalism that is quite general yet still supports all the latter geometric features, so that there is a sense in which a smooth space carries a fixed geometry once it is equipped with a chosen connection.  Thus we are led to define the notion of \textit{geometric space} in a tangent category $\C$, as an object equipped with a connection (\ref{def:geo}).  In view of the Auslander-Markus theorem (\ref{thm:aus-mar}), the category of affine manifolds has a natural generalization in an arbitrary tangent category, namely the category of geometric spaces whose associated connection is flat and torsion-free.  Thus we may pursue certain of the themes of Jubin's thesis within the category of \textit{affine geometric spaces} in $\C$, which we now define:

\begin{definition}\label{def:geo}
Let $(\C,T)$ be a tangent category.
\begin{itemize}
\item A \textbf{geometric space} in $\C$ is a pair $(M,K)$ in which $M$ is an object of $\C$
and $K:T^2M \rightarrow TM$ is a connection on $M$ (\ref{def:acx-v}).
\item A geometric space $(M,K)$ is \textbf{flat} (resp. \textbf{torsion-free}) if its associated connection $K$ is so (see \ref{defn:flat} and \ref{defn:torsionFree}).  
\item An \textbf{affine geometric space} in $\C$ is a geometric space $(M,K)$ that is both flat and torsion-free. 
\item A map of geometric spaces $f\colon(M,K)\rarr(M',K')$ is a map $f\colon M\rarr M'$ in $\C$ such that the following diagram commutes:
$$\xymatrix{T^2M \ar[d]_{T^2f} \ar[r]^{K} & TM \ar[d]^{Tf} \\ T^2M' \ar[r]_{K'} &TM'}$$
\item We write $\Geom(\C,T)$ to denote the category of geometric spaces, with the above morphisms.  We denote by $\GeomFlat(\C,T)$ and $\GeomTf(\C,T)$ the full subcategories of $\Geom(\C,T)$ consisting of the flat and torsion-free geometric spaces, respectively.
\item We write $\Aff(\C,T)$ to denote the full subcategory of $\Geom(\C,T)$ whose objects are the affine geometric spaces.
\end{itemize}
\end{definition}

\begin{example}
By Example \add{5.7} of \cite{CC4}, any differential object has a canonical choice of connection, given by the formula
\[ K = \<T(\hat{p})\hat{p},pp\>  \]
(its associated horizontal connection is $H = \<!\zq,\pi_0\hat{p},\pi_1\hat{p},\pi_1p\>$).  This connection is flat and torsion-free (see the discussion after Proposition 3.16 and Example 3.21 in \cite{CC4}).  Thus, any differential object has a canonical choice of connection to make it into an affine geometric space.  

Moreover, recall that if $(A,\hat{p}_A)$ and $(B,\hat{p}_B)$ are differential objects, then a \textbf{linear map} between such objects consists of a map $f: A \to B$ such that $T(f)\hat{p}_B = \hat{p}_Af$.  It is then easy to \add{show that }such a map is also a map between the corresponding affine geometric spaces, i.e., a map in $\Aff(\C,T)$.  
\end{example}

\begin{example} We note that any Riemannian manifold whose canonical connection (the Levi-Civita connection)
is flat is automatically affine, since the Levi-Civita connection is always torsion-free. This gives us access to a wide variety
of further examples. 
\end{example}

\add{Recall that the morphisms in the category $\Aff$ of affine manifolds are the locally affine maps (\S \ref{Aff})}.  The following result shows that \add{these are the same as maps that preserve the associated connections}:

\begin{proposition}\label{prop:affineMaps}
\add{A smooth map $f:M \rightarrow M'$ between affine manifolds $M$ and $M'$ is locally affine if and only if $f:(M,K) \rightarrow (M',K')$ is a morphism of geometric spaces in the tangent category $(\Mf,T)$ of smooth manifolds, where $K$ and $K'$ are the associated connections (\ref{thm:aus-mar}).  Consequently, the category $\Aff$ of affine manifolds is equivalent to the category $\Aff(\Mf,T)$ of affine geometric spaces in $(\Mf,T)$.}
\end{proposition}
\begin{proof}
\add{In view of the Auslander-Markus Theorem (}\ref{thm:aus-mar}\add{), it suffices to prove the first statement above, concerning a given map $f$.}\add{  By definition, $f:(M,K) \rightarrow (M',K')$ is a morphism of geometric spaces iff $T^2(f)K' = KT(f)$.}  Since we are dealing with smooth manifolds, it suffices to know that this equality holds in each of the \add{given} charts for $M$.  However, \add{by }\cite[Thm. 1]{AuMar}\add{, the Christoffel symbols of the associated connections $K$ and $K'$ are identically zero on each of the given charts}, which means that in each \add{of these charts}, the connection $K$ (and similarly $K'$) takes the particular form
	\[ (x,v,w,a) \mapsto (x,a) \]
(see \cite[Example 3.6.1]{CC4} for the relationship of the Christoffel symbols to the vertical connection $K$).

For the remainder of the proof we will work locally; that is, we consider a pair of charts $U \cong V \subseteq \mathbb{R}^n$ and $U' \cong V' \subseteq \mathbb{R}^m$ in $M$ and $M'$, respectively, and consider the restriction of $f$ to $U \cap f^{-1}(U')$.  For simplicity, we will also simply consider the case when $m=1$.  Recall that in a chart $U \subseteq R^n$, for a point $(x,v) \in TU$,
	\[ T(f)(x,v) = (f(x),D(f)(x,v)) \]
where $D(f)(x,v)$ is the directional derivative of $f$ at $x$ in the direction of $v$, and as a result
	\[ T^2(f)(x,v,w,a) = (f(x),D(f)(x,v),D(f)(x,w),D(D(f))((x,v),(w,a))). \]
Thus, by the form \add{that} $K$ and $K'$ take, we have 	
	\[ T(f)(K(x,v,w,a)) = T(f)(x,a) = (f(x),D(f)(x,a)) \]
while
	\[ K'(T^2(f)(x,v,w,a)) = (f(x),D(D(f))((x,v),(w,a))). \]
However, by definition of the directional derivative,
	\[ D(D(f))((x,v),(w,a)) = \frac{\add{\partial[}D(f)(x,v)\add{]}}{\add{\partial} x}\cdot w + \frac{\add{\partial[}D(f)(x,v)\add{]}}{\add{\partial} v}\cdot a = w^T \cdot H(f)(x)\cdot v + D(f)(x,a) \]
where $H$ is the Hessian of $f$, ie., the matrix of second partial derivatives of $f$.  

Thus the two terms are equal if and only if 
	\[ w^T \cdot H(f)(x) \cdot v  = 0 \]
for all $x,v,w$.   But this is true if and only if each second partial derivative of $f$ at $x$ is equal to 0.  In other words, the map $f$ is connection-preserving if and only if in each local affine chart, and for each $i$, $\frac{\add{\partial}f}{\add{\partial}x_i}$ has each of its partial derivatives equal to 0.  
\end{proof}

Every map in $\Geom(\C,T)$ necessarily preserves the associated horizontal connection (\ref{thm:pres-hcx}).  We will prove this by means of the following proposition:

\begin{proposition}\label{propConnPreservingMaps}
Suppose $(f,g): {\sf q} \to {\sf q'}$ is a linear map between differential bundles with connections $(K,H)$ and $(K',H')$.  Then 
	\[ T(f)K' = Kf \Leftrightarrow (Tg \times f)H' = HT(f). \]
\end{proposition}
\begin{proof}
Suppose $T(f)K' = Kf$.  We have
\begin{eqnarray*}
\<K',p\>\mu' + U' H' & = & 1 \\
T(f)\<K',p\>\mu' + T(f)U' H' & = & T(f) \\
\<K,p\>(f \times f)\mu' + U(Tg \times f)H' & = & T(f) \mbox{ (by Lemma 4.2 of \cite{CC4})} \\
\<K,p\>\mu T(f) + U(Tg \times f)H' & = & T(f) \mbox{ (by Lemma 2.17 of \cite{CC3})} \\
H \<K,p\>\mu T(f) + H U(Tg \times f)H' & = & H T(f) \\
0 + (Tg \times f)H' & = & HT(f)
\end{eqnarray*}
as required.  For the other direction, suppose $(Tg \times f)H' = HT(f)$.  Then we have
\begin{eqnarray*}
\<K,p\>\mu  + UH & = & 1 \\
\<K,p\>\mu T(f) + UHT(f) & = & T(f) \\
\<K,p\>(f \times f)\mu' + U(Tg \times f)H' & = & T(f) \mbox{ (by Lemma 4.2 of \cite{CC4})} \\
\<K,p\>(f \times f)\mu' + T(f)U'H' & = & T(f) \mbox{ (by Lemma 2.17 of \cite{CC3})} \\
\<K,p\>(f \times f)\mu' \<K',p\> + T(f)U'H'\<K',p\> & = & T(f)\<K',p\> \\
\<K,p\>(f \times f) & = & T(f)\<K',p\> 
\end{eqnarray*}
so that by taking the first projection of both sides, $Kf = T(f)K'$, as required.
\end{proof}

\begin{corollary}\label{thm:pres-hcx}
If $M$ and $M'$ have connections $(K,H)$ and $(K',H')$, then for any map $f: M \to M'$, 
	\[ T^2(f)K' = KT(f) \Leftrightarrow T_2(f)H' = HT^2(f). \]
\end{corollary}
\begin{proof}
By Proposition 2.4 of \cite{CC4}, the pair $(T(f),f)$ is a linear bundle morphism from the tangent bundle of $M$ to the tangent bundle of $M'$.  Then applying the previous result (\ref{propConnPreservingMaps}) to this linear bundle morphism gives the desired result.
\end{proof}

\section{Lifting tangent structure to the geometric categories} 

The main result we would like to prove in this section is the following:

\begin{theorem}\label{thm:aff-tcat}
Let $(\C,T)$ be a tangent category. There is a functor $T_*\colon{\sf Aff}(\C,T)\rarr{\sf Aff}(\C,T)$
given on objects as follows:
\[(M,K)\mapsto (TM,T(c)cT(K)c)\]
This functor makes ${\sf Aff}(\C,T)$ a tangent category.
\end{theorem}

However, some of the structures that arise in proving this result will also lead us to prove an interesting alternate characterization of flat torsion-free connections (Theorem \ref{propFTFConditions}).  

Before proving the result above, we will pause to consider where the above formula comes from.  Given any \textit{strong morphism of tangent categories} $F:\C \rightarrow \C'$, in the sense of \ref{def:str-mor} below, we shall show that $F$ sends a connection on $M$ in $\C$ to a connection on $FM$ in $\C'$ (\ref{thm:sm-acx}).  In particular, the tangent functor $T:\C \rightarrow \C$ is a strong morphism of tangent categories when equipped with the transformation $c$, so by applying $T$ to a connection $K$ on $M$ and composing with a few instances of $c$ we obtain an associated connection $T(c)cT(K)c$ on $TM$.

\begin{definition}[{\cite{CC1}}]\label{def:str-mor}
Given tangent categories $(\C,T,p,0,+,\ell,c)$ and $(\C',T',p',0',+',\ell',c')$, a \textbf{morphism of tangent categories} is a functor $F:\C \rightarrow \C'$ equipped with a natural transformation $\alpha = \alpha^F:FT \rightarrow T'F$ such that $F$ preserves all the pullbacks that are required to exist as part of the tangent structure on $\C$, and such that
$$\alpha p'_F = F(p),\;\;\;\; F(0)\alpha = 0'_F,\;\;\;\; F(+)\alpha = \alpha_2+'_F,$$
$$F(\ell)\alpha^{[2]} = \alpha\ell'_F,\;\;\;\; F(c)\alpha^{[2]} = \alpha^{[2]}c'_F.$$
Here $\alpha^{[2]} = \alpha_TT'(\alpha):FT^2 \rightarrow {T'}^2F$, and $\alpha_2:FT_2 \rightarrow T'_2F$ is the natural transformation whose components
$$F(TM \times_M TM) = FTM \times_{FM} FTM \longrightarrow T'FM \times_{FM} T'FM$$
at each object $M$ are induced by $\alpha_M$ on each factor.  A morphism of tangent categories $(F,\alpha)$ is said to be \textbf{strong} if $\alpha$ is invertible, and \textbf{strict} if $\alpha$ is an identity. 
\end{definition}

Morphisms of tangent categories are the arrows of a category \cite[Def. 2.7]{CC1}, in which the composite of morphisms $(F,\alpha^F):(\C,T) \rightarrow (\C',T')$ and $(G,\alpha^G):(\C',T') \rightarrow (\C'',T'')$ is $(GF,\alpha^F * \alpha^G)$, where we define $\alpha^F * \alpha^G = G(\alpha^F)\alpha^G_F:GFT \rightarrow T''GF$.

\subsection{}\label{thm:sm-func-db} A strong morphism of tangent categories $F:(\C,T) \rightarrow (\C',T')$ sends each differential bundle ${\sf q} = (q:E \rightarrow M,\pq,\zq,\lambda)$ in $\C$ to a differential bundle 
$$F({\sf q}) = (F(q),F(\pq),F(\zq),F(\lambda)\alpha^F_E)$$
in $\C'$, and this assignment is functorial with respect to linear morphisms of differential bundles \cite[Prop. 4.22]{CC3}.

\begin{lemma}\label{thm:isos-db-sm}
Let $(F,\alpha):(\C,T) \rightarrow (\C',T')$ be a strong morphism of tangent categories.  Then for each object $M$ of $\C$ we have linear isomorphisms of differential bundles as follows:
\begin{enumerate}
\item $(\alpha_M,1_{FM})\;:\;F({\sf p}_M) \xrightarrow{\sim} {\sf p}'_{FM}$,
\item $(\alpha^{[2]}_M,\alpha_M)\;:\;F({\sf p}_{TM}) \xrightarrow{\sim} {\sf p}'_{T'FM}$.
\item $(\alpha^{[2]}_M,\alpha_M)\;:\;FT({\sf p}_M) \xrightarrow{\sim} T'({\sf p}'_{FM})$,
\end{enumerate}
\end{lemma}
\begin{proof}
The fact that 1 is a linear bundle morphism follows immediately from the axioms in \ref{def:str-mor}.  In particular, $(\alpha_{TM},1_{FTM}):F({\sf p}_{TM}) \rightarrow {\sf p}'_{FTM}$ is a linear isomorphism, but we also know that the isomorphism $\alpha_M:FTM \rightarrow T'FM$ induces a linear isomorphism $(T'(\alpha_M),\alpha_M):{\sf p}'_{FTM} \rightarrow {\sf p}'_{T'FM}$, and by composition we obtain the linear isomorphism needed in 2.  Also, the axioms for a tangent category yield linear isomorphisms $(c_M,1_{TM}):T({\sf p}_M) \rightarrow {\sf p}_{TM}$ and $(c'_{FM},1_{T'FM}):T'({\sf p}'_{FM}) \rightarrow {\sf p}'_{T'FM}$, so by the functoriality in \ref{thm:sm-func-db}, and the fact that $c'$ is an involution, we obtain a linear composite
$$FT({\sf p}_M) \xrightarrow{(F(c_M),1_{FTM})} F({\sf p}_{TM}) \xrightarrow{(\alpha^{[2]}_M,\alpha_M)} {\sf p}'_{T'FM} \xrightarrow{(c'_{FM},1_{T'FM})} T'({\sf p}'_{FM}),$$
which can be expressed equally as $(\alpha^{[2]}_M,\alpha_M)$ since $F(c_M)\alpha^{[2]}_Mc'_{FM} = \alpha^{[2]}_Mc'_{FM}c'_{FM} = \alpha^{[2]}_M$ by \ref{def:str-mor}.
\end{proof}

\begin{proposition}\label{thm:sm-acx}
Let $F:(\C,T) \rightarrow (\C',T')$ be a strong morphism of tangent categories, and let $K:T^2M \rightarrow TM$ be a connection on $M$ in $\C$.  Then the composite
$$K_F = \left({T'}^2FM \xrightarrow{\alpha^{[-2]}_M} FT^2M \xrightarrow{F(K)} FTM \xrightarrow{\alpha_M} T'FM\right)$$
is a connection on $FM$, where $\alpha = \alpha^F$ and $\alpha^{[-2]} = (\alpha^{[2]})^{-1}$. 
\end{proposition}
\begin{proof}
By \ref{thm:acx} and \ref{def:acx-v}, it suffices to show that $K_F$ is a vertical connection and that
\begin{equation}\label{eq:sm-fp}
\xymatrix{
                 & {T'}^2FM \ar[dl]_{T'(p'_{FM})} \ar[d]|{p'_{T'FM}} \ar[dr]^{K_F} & \\
T'FM \ar[dr]_{p'_{FM}} & T'FM \ar[d]|{p'_{FM}}                                 & T'FM \ar[dl]^{p'_{FM}}\\
                 & FM                                              &
}
\end{equation}
is a fibre product diagram in $\C'$.

By \ref{thm:acx}, we know that the diagram \eqref{eq:acx-fp} presents $T^2M$ as a third fibre power of $p_M:TM \rightarrow M$.  But $F$ preserves finite fibre powers of $p_M$, so $F$ sends the diagram \eqref{eq:acx-fp} to a fibre product diagram
$$
\xymatrix{
                 & FT^2M \ar[dl]_{FT(p_M)} \ar[d]|{F(p_{TM})} \ar[dr]^{F(K)} & \\
FTM \ar[dr]_{F(p_M)} & FTM \ar[d]|{F(p_M)}                                 & FTM \ar[dl]^{F(p_M)}\\
                 & FM                                              &
}
$$
in $\C'$.  By composing with the isomorphism $\alpha^{[-2]}_M:{T'}^2FM \rightarrow FT^2M$, we find that the morphisms
$$\alpha^{[-2]}_MFT(p_M),\;\alpha^{[-2]}_MF(p_{TM}),\;\alpha^{[-2]}_MF(K)\;\;:\;\;{T'}^2FM \rightarrow FTM$$
present ${T'}^2FM$ as a third fibre power of $F(p_M)$ in $\C'$.  But $F(p_M) = \alpha_Mp'_{FM}:FTM \rightarrow FM$ since $(F,\alpha)$ is a morphism of tangent categories, so since $\alpha$ is an isomorphism we deduce that the composites
$$f_1 := \alpha^{[-2]}_MFT(p_M)\alpha_M,\;f_2 := \alpha^{[-2]}_MF(p_{TM})\alpha_M,\;f_3 := \alpha^{[-2]}_MF(K)\alpha_M\;\;:\;\;{T'}^2FM \rightarrow T'FM$$
present ${T'}^2FM$ as a third fibre power of $p'_{FM}:T'FM \rightarrow FM$ in $\C'$.

Hence, in order to show that \eqref{eq:sm-fp} is a fibre product diagram, it suffices to show that $f_1 = T'(p'_{FM}), f_2 = p'_{T'FM}, f_3 = K_F$.  The third of these equations holds by the definition of $K_F$.  The first two equations also hold, because
$$\alpha^{[2]}_MT'(p'_{FM}) = \alpha_{TM}T'(\alpha_M)T'(p'_{FM}) = \alpha_{TM}T'F(p_M) = FT(p_M)\alpha_M$$
$$\alpha^{[2]}_Mp'_{T'FM} = \alpha_{TM}T'(\alpha_M)p'_{T'FM} = \alpha_{TM}p'_{FTM}\alpha_M = F(p_{TM})\alpha_M$$
since $(F,\alpha)$ is a strong morphism and $p'$ is natural.

Now it suffices to show that $K_F$ is a vertical connection on ${\sf p}'_{FM}$.  Firstly, $K_F$ is a retraction of $\ell'_{FM}:T'FM \rightarrow {T'}^2FM$ since $\ell'_{FM}K_F = \ell'_{FM}\alpha^{[-2]}_MF(K)\alpha_M = \alpha^{-1}_MF(\ell_M)F(K)\alpha_M = \alpha^{-1}_MF(\ell_MK)\alpha_M = 1_{T'FM}$, because $\ell_MK = 1_{TM}$.  Hence it suffices to show that $(K_F,p'_{FM}):T'({\sf p}'_{FM}) \rightarrow {\sf p}'_{FM}$ and $(K_F,p'_{FM}):{\sf p}'_{T'FM}\rightarrow {\sf p}'_{FM}$ are linear morphisms of differential bundles.  But since $K$ is a vertical connection on ${\sf p}_M$, we know that $(K,p_M):T({\sf p}_M) \rightarrow {\sf p}_M$ and $(K,p_M):{\sf p}_{TM} \rightarrow {\sf p}_M$ are linear morphisms of differential bundles and so, by \ref{thm:sm-func-db}, are sent by $F$ to linear morphisms of differential bundles
$$(F(K),F(p_M)):FT({\sf p}_M) \rightarrow F({\sf p}_M)$$
$$(F(K),F(p_M)):F({\sf p}_{TM}) \rightarrow F({\sf p}_M)\;.$$
Hence by composition with the linear isomorphisms in \ref{thm:isos-db-sm} we obtain linear bundle morphisms
$$T'({\sf p}'_{FM}) \xrightarrow{(\alpha^{[-2]}_M,\alpha^{-1}_M)} FT({\sf p}_M) \xrightarrow{(F(K),F(p_M))} F({\sf p}_M) \xrightarrow{(\alpha_M,1_{FM})} {\sf p}'_{FM}$$
$${\sf p}'_{T'FM} \xrightarrow{(\alpha^{[-2]}_M,\alpha^{-1}_M)} F({\sf p}_{TM}) \xrightarrow{(F(K),F(p_M))} F({\sf p}_M) \xrightarrow{(\alpha_M,1_{FM})} {\sf p}'_{FM}.$$
But the pair $(K_F,p'_{FM})$ underlies each of these two composites.
\end{proof}

\begin{proposition}\label{thm:sm-hcx}
Let $(F,\alpha):(\C,T) \rightarrow (\C',T')$ be a strong morphism of tangent categories, and let $K:T^2M \rightarrow TM$ be a connection on an object $M$ of $\C$.  Then the horizontal connection associated to the connection $K_F$ is the composite
$$H_F = \left(T'_2FM \xrightarrow{\alpha^{-1}_2} FT_2M \xrightarrow{F(H)} FT^2M \xrightarrow{\alpha^{[2]}_M} {T'}^2FM \right),$$
recalling that $\alpha_2$ and $\alpha^{[2]}$ are defined in \ref{def:str-mor}.
\end{proposition}
\begin{proof}
By \ref{thm:hcx-for-evcx}, it suffices to establish the following equations
$$H_FT'(p'_{FM}) = \pi_0,\;\;H_Fp'_{T'FM} = \pi_1,\;\;H_FK_F = p'_20'_{FM}\;\;:\;\;T'_2FM \longrightarrow T'FM,$$
but we know that $H$ satisfies the analogous equations $HT(p_M) = \pi_0$, $Hp_{TM} = \pi_1$, $HK = p_20_M$.  Hence we compute that
$$
\begin{array}{llll}
H_FT'(p'_{FM}) & = & \alpha_2^{-1}F(H)\alpha^{[2]}_MT'(p'_{FM}) & \\
                       & = & \alpha_2^{-1}F(H)FT(p_M)\alpha_M & \text{(by \ref{thm:isos-db-sm}(3))}\\
                       & = & \alpha_2^{-1}F(\pi_0)\alpha_M &\\
                       & = & \pi_0\alpha_M^{-1}\alpha_M & \text{(by the definition of $\alpha_2$)}\\
                       & = & \pi_0 &
\end{array}
$$
$$
\begin{array}{llll}
H_Fp'_{T'FM}   & = & \alpha_2^{-1}F(H)\alpha^{[2]}_Mp'_{T'FM} &\\
                       & = & \alpha_2^{-1}F(H)F(p_{TM})\alpha_M & \text{(by \ref{thm:isos-db-sm}(2))}\\
                       & = & \alpha_2^{-1}F(\pi_1)\alpha_M &\\
                       & = & \pi_1\alpha_M^{-1}\alpha_M & \text{(by the definition of $\alpha_2$)}\\
                       & = & \pi_1 &\\
H_FK_F & = & \alpha_2^{-1}F(H)\alpha^{[2]}_M\alpha^{[-2]}_MF(K)\alpha_M &\\
                       & = & \alpha_2^{-1}F(H)F(K)\alpha_M &\\
                       & = & \alpha_2^{-1}F(p_2)F(0_M)\alpha_M &\\
                       & = & \alpha_2^{-1}F(p_2)0'_{FM} & \text{(by \ref{def:str-mor})}\\
                       & = & p'_20'_{FM} &
\end{array}
$$
since it follows readily from \ref{def:str-mor} that $\alpha_2^{-1}F(p_2) = p'_2$.
\end{proof}

\begin{proposition}\label{thm:sm-tf-fl}
Let $F:(\C,T) \rightarrow (\C',T')$ be a strong morphism of tangent categories, and let $K:T^2M \rightarrow TM$ be a connection on an object $M$ of $\C$.
\begin{enumerate}
\item If $K$ is torsion-free, then $K_F$ is a torsion-free connection on $FM$.
\item If $K$ is flat, then $K_F$ is a flat connection on $FM$.
\end{enumerate}
\end{proposition}
\begin{proof}
If $K$ is torsion-free, i.e. $c_MK = K$, then
$$c'_{FM}K_F = c'_{FM}\alpha^{[-2]}_MF(K)\alpha_M = \alpha^{[-2]}_MF(c_M)F(K)\alpha_M = \alpha^{[-2]}_MF(K)\alpha_M = K_F$$
by \ref{def:str-mor}.  Suppose that $K$ is flat, i.e. $c_{TM}T(K)K = T(K)K:T^3M \rightarrow TM$.  Then
$$
\begin{array}{lll}
T'(K_F)K_F & = & T'(\alpha^{[-2]}_M)T'F(K)T'(\alpha_M)T'(\alpha^{-1}_M)\alpha^{-1}_{TM}F(K)\alpha_M\\
                           & = & T'(\alpha^{[-2]}_M)T'F(K)\alpha^{-1}_{TM}F(K)\alpha_M\\
                           & = & T'(\alpha^{[-2]}_M)\alpha^{-1}_{T^2M}FT(K)F(K)\alpha_M\\
                           & = & T'(\alpha^{[-2]}_M)\alpha^{-1}_{T^2M}F(T(K)K)\alpha_M\\
                           & = & {T'}^2(\alpha^{-1}_M)T'(\alpha^{-1}_{TM})\alpha^{-1}_{T^2M}F(T(K)K)\alpha_M\\
                           & = & {T'}^2(\alpha^{-1}_M)\alpha^{[-2]}_{TM}F(T(K)K)\alpha_M
\end{array}
$$
by the naturality of $\alpha^{-1}$ and the definition of $\alpha^{[-2]}$.  Hence
$$
\begin{array}{lll}
c'_{T'FM}T'(K_F)K_F & = & c'_{T'FM}{T'}^2(\alpha^{-1}_M)\alpha^{[-2]}_{TM}F(T(K)K)\alpha_M\\
                                    & = & {T'}^2(\alpha^{-1}_M)c'_{FTM}\alpha^{[-2]}_{TM}F(T(K)K)\alpha_M\\
                                    & = & {T'}^2(\alpha^{-1}_M)\alpha^{[-2]}_{TM}F(c_{TM})F(T(K)K)\alpha_M\\
                                    & = & {T'}^2(\alpha^{-1}_M)\alpha^{[-2]}_{TM}F(T(K)K)\alpha_M\\
                                    & = & T'(K_F)K_F.
\end{array}
$$
by \ref{def:str-mor} and the naturality of $c'$. 
\end{proof}

\begin{proposition}\label{thm:sm-aff-func}
Every strong morphism of tangent categories $F:(\C,T) \rightarrow (\C',T')$ induces a functor
$$F_*\;:\:\Geom(\C,T) \rightarrow \Geom(\C',T')$$
given on objects by $(M,K) \mapsto (FM,K_F)$ and on morphisms by $f \mapsto F(f)$.  The analogous claims hold with each of $\GeomFlat$, $\GeomTf$, and $\Aff$ replacing $\Geom$, and in each case we shall denote the resulting functor also by $F_*$.
\end{proposition}
\begin{proof}
By \ref{thm:sm-acx} and \ref{thm:sm-tf-fl}, it suffices to show that if $f:(M,K) \rightarrow (M',K')$ is a morphism in $\Geom(\C,T)$, then $F(f):(FM,K_F) \rightarrow (FM',K'_F)$ is a morphism in $\Geom(\C',T')$.  But this follows immediately from the definitions, using the naturality of $\alpha$ and $\alpha^{[-2]}$.
\end{proof}

\subsection{}\label{par:2cat-tan}
Given morphisms of tangent categories $F,G:(\C,T) \rightarrow (\C',T')$, a \textbf{tangent transformation} $\phi:F \Rightarrow G$ is a natural transformation such that $\alpha^FT'(\phi) = \phi_T\alpha^G$ \cite[Def. 4.18]{CC3}.  It is straightforward to show that tangent transformations are closed under vertical composition and are closed under whiskering with morphisms of tangent categories.  Hence, in view of \ref{def:str-mor}, we obtain a 2-category $\Tan$ whose objects are tangent categories, whose 1-cells are strong morphisms, and whose 2-cells are tangent transformations.

\begin{lemma}\label{thm:ttr-mor-conn}
Let $F,G:(\C,T) \rightarrow (\C',T')$ be strong morphisms of tangent categories, and let $\phi:F \Rightarrow G$ be a tangent transformation.  Then for any connection $K$ on an object $M$ of $\C$, the component $\phi_M$ underlies a morphism
\begin{equation}\label{eq:comp-ttr-morcx}\phi_M\;:\;(FM,K_F) \longrightarrow (GM,K_G)\end{equation}
in $\Geom(\C',T')$.  Further, there is a natural transformation
$$\phi_*\;:\;F_* \Rightarrow G_*\;:\;\Geom(\C,T) \rightarrow \Geom(\C',T')$$
whose component at each object $(M,K)$ of $\Geom(\C,T)$ is the morphism \eqref{eq:comp-ttr-morcx}.
\end{lemma}
\begin{proof}
The first claim follows immediately from the naturality of $\alpha^F$ and $\alpha^G$, and the second is immediate.
\end{proof}

\begin{theorem}\label{thm:geom-2func-to-cat}
There are 2-functors
$$\Geom,\;\GeomFlat,\;\GeomTf,\;\Aff\;:\;\Tan \rightarrow \Cat$$
from the 2-category $\Tan$ of tangent categories (\ref{par:2cat-tan}) to the 2-category $\Cat$ of categories, sending each tangent category $(\C,T)$ to $\Geom(\C,T)$, $\GeomFlat(\C,T)$, $\GeomTf(\C,T)$, and $\Aff(\C,T)$, respectively.  These 2-functors are given on 1-cells by \ref{thm:sm-aff-func} and on 2-cells by \ref{thm:ttr-mor-conn}.
\end{theorem}
\begin{proof}
By employing the definitions, as well as the middle-interchange law for $\Cat$, it is straightforward to verify the needed functoriality on 1-cells.  Functoriality with respect to vertical composition of 2-cells is immediate, as is the preservation of whiskering by $\Geom$ (and hence by the others).
\end{proof}

We now apply this theorem in order to show that $\Geom(\C,T)$ and $\Aff(\C,T)$ are tangent categories, by way of the following general lemma.

\begin{lemma}\label{thm:tstr-sm-tt}
Let $(\C,T,+,0,\ell,c)$ be a tangent category.
\begin{enumerate}
\item \textnormal{\cite{CC3}} $(T,c):(\C,T) \rightarrow (\C,T)$ is a strong morphism of tangent categories. 
\item \textnormal{\cite{CC3}} $(T_n,c_n):(\C,T) \rightarrow (\C,T)$ is a strong morphism of tangent categories for each natural number $n$, where $c_n:T_nT \rightarrow TT_n$ is the unique morphism such that $c_nT(\pi_i) = \pi_ic$ for each $i = 0,...,n-1$ when we write $\pi_i:T_n \rightarrow T$ to denote the projection. 
\item The following are tangent transformations
$$p:(T,c) \Longrightarrow (1,1_T),\;\;\;\;+:(T_2,c_2) \Longrightarrow (T,c),\;\;\;\;0:(1,1_T) \Longrightarrow (T,c),$$
$$\ell:(T,c) \Longrightarrow (T,c)^2,\;\;\;\;c:(T,c)^2 \Longrightarrow (T,c)^2,\;\;\;\;\pi_i:(T_n,c_n) \Longrightarrow (T,c)$$
for all natural numbers $n,i$ with $i < n$, where $(T,c)^2 = (T,c) \circ (T,c) = (T^2, c * c)$ is the composite 1-cell in $\Tan$, where $c * c = T(c)c_T:T^3 \rightarrow T^3$ (\ref{def:str-mor}).
\end{enumerate}
\end{lemma}
\begin{proof}
It suffices to prove 3.  Firstly, $p,0,\ell,c$ are tangent transformations since $cT(p) = p_T = p_T1_T$, $1_TT(0) = T(0) = 0_Tc$, $\ell_T(c * c) = \ell_TT(c)c_T = cT(\ell)$, and $c_T(c * c) = c_TT(c)c_T = T(c)c_TT(c) = (c * c)T(c)$, by the axioms for a tangent category.  The definition of $c_n$ immediately entails that each $\pi_i$ is a tangent transformation.  With regard to $+$, one of the axioms for a tangent category entails that $(c,1_{TM}):{\sf p}_{TM} \rightarrow T({\sf p}_M)$ is an additive bundle morphism, so by the definition of $c_2$ we deduce that
$$
\xymatrix{
T_2TM \ar[d]_{+_{TM}} \ar[r]^{(c_2)_M} & TT_2M \ar[d]^{T(+_M)}\\
TTM \ar[r]_{c_M}                       & TTM
}
$$
commutes.
\end{proof}

\begin{corollary}\label{thm:strmr-aff-conn}
Given a tangent category $(\C,T)$, we can apply the 2-functor $\Aff:\Tan \rightarrow \Cat$ to the 1-cells $T,T_n:(\C,T) \rightarrow (\C,T)$ and 2-cells $p,+,0,\ell,c$ in $\Tan$ in order to obtain functors
$$T_*,\;(T_n)_*\;:\;\Aff(\C,T) \rightarrow \Aff(\C,T)$$
and natural transformations
$$p_*:T_* \Longrightarrow 1,\;\;\;\;+_*:(T_2)_* \Longrightarrow T_*,\;\;\;\;0_*:1 \Longrightarrow T_*$$
$$\ell_*:T_* \Longrightarrow T_*^2,\;\;\;\;c_*:T_*^2 \Longrightarrow T_*^2,\;\;\;\;(\pi_i)_*:(T_n)_* \Longrightarrow T_*$$
for all natural numbers $n,i$ with $i < n$.  We can similarly apply $\Geom,\GeomFlat,\GeomTf$ to the same data in order to obtain endofunctors and natural transformations, for which we employ the same notations.
\end{corollary}

In order to show that \ref{thm:strmr-aff-conn} yields a tangent structure on $\Aff(\C,T)$, we shall need certain finite limits in the latter category.  To this end we shall employ the following:

\begin{lemma}\label{thm:refl-lim-cx}
Let $D:\J \rightarrow \Geom(\C,T)$ be a functor, and let $\pi = (\pi_j:L \rightarrow Dj)_{j \in \J}$ be a cone on $D$.  Writing $U:\Geom(\C,T) \rightarrow \C$ for the forgetful functor, suppose that $\pi$ is sent by $U$ to a limit cone for $UD$ that is preserved by $T^k$ for each natural number $k$.  Then
\begin{enumerate}[(i)] 
	\item $\pi$ is a limit cone for $D$, 
	\item this limit is preserved by each of the endofunctors $T^k_*$ on $\Geom(\C,T)$.
\end{enumerate}
\end{lemma}
\begin{proof}
Let us write $Dj = (UDj,K_j)$ for each object $j$ of $\J$, and write $L = (L_0,L_1)$.  Given any cone $(f_j:(M,K) \rightarrow Dj)_{j \in \J}$ on $D$, we know that $(f_j:M \rightarrow UDj)_{j \in \J}$ is a cone on $UD$ and hence induces a morphism $f:M \rightarrow L_0$ in $\C$.  For each $j \in \ob\J$ we compute that
$$KT(f)T(\pi_j) = KT(f_j) = T^2(f_j)K_j = T^2(f)T^2(\pi_j)K_j = T^2(f)L_1T(\pi_j)$$
since $f_j$ and $\pi_j$ are morphisms in $\Geom(\C,T)$, so since $(T(\pi_j):TL_0 \rightarrow TUDj)_{j \in \J}$ is a limit cone in $\C$ we deduce that $KT(f) = T^2(f)L_1$.  Hence $f:(M,K) \rightarrow L$ is a morphism in $\Geom(\C,T)$.  Thus (i) is proved.

For each natural number $k$, we know that $T_*^k(\pi) = (T_*^k(\pi_j))_{j \in \J}$ is a cone on the diagram $T_*^kD$ and is sent by $U$ to a limit cone $(T^k(\pi_j):T^kL_0 \rightarrow T^kUDj)_{j \in \J}$ for the diagram $UT_*^kD = T^kUD:\J \rightarrow \C$.  Further, the latter limit is preserved by $T^{k'}$ for each natural number $k'$, so we can apply (i) to the cone $T_*^k(\pi)$ in order to deduce that $T_*^k(\pi)$ is a limit cone for $T_*^kD$.
\end{proof}

The preceding lemma immediately entails the following:

\begin{lemma}\label{thm:fps-in-aff}
Let $(\C,T)$ be a tangent category.  Then for each natural number $n$ and each object $M$ of $\Aff(\C,T)$, the morphisms ${(\pi_i)_*}_M:(T_n)_*M \rightarrow T_*M$ present $(T_n)_*M$ as an $n$-th fibre power of ${p_*}_M:T_*M \rightarrow M$ in $\Aff(\C,T)$, and this fibre power is preserved by $T_*^k:\Aff(\C,T) \rightarrow \Aff(\C,T)$ for each natural number $k$.  The analogous claims hold with each of $\Geom,\GeomFlat,\GeomTf$ in place of $\Aff$.
\end{lemma}

\begin{theorem}\label{thm:aff-conn-tcats}
Let $(\C,T)$ be a tangent category.  Then each of the categories $\Geom(\C,T)$, $\GeomFlat(\C,T)$, $\GeomTf(\C,T)$, and $\Aff(\C,T)$ is a tangent category when equipped with its endofunctor $T_*$ and natural transformations $p_*,0_*,+_*,\ell_*,c_*$ as defined in \ref{thm:strmr-aff-conn}.
\end{theorem}
\begin{proof}
All of the needed structure is furnished by \ref{thm:strmr-aff-conn} and \ref{thm:fps-in-aff}.  This structure satisfies the equational axioms for a tangent category, by the 2-functoriality of $\Geom$, $\GeomFlat$, $\GeomTf$, and $\Aff:\Tan \rightarrow \Cat$, so it remains only to verify the \textit{universality of the vertical lift} \cite[Def. 2.1]{CC3}.  It suffices to treat the case of $\Geom(\C,T)$, from which the needed property of each of the other categories then follows.  For each object $M$ of the category $\D = \Geom(\C,T)$, we must show that a particular commutative square $S$ in $\D$ is a pullback that is preserved by each $T_*^n$ \cite[Def. 2.1]{CC3}, where $S$ is defined in terms of the (candidate) tangent structure on $\D$.  But the square $S$ is sent by the forgetful functor $U:\D \rightarrow \C$ to the similarly defined square in $\C$, which we know is a pullback in $\C$ that is preserved by each $T^n$.  Hence by \ref{thm:refl-lim-cx} we deduce that $S$ is a pullback square in $\D$ that is preserved by each $T^n_*$.
\end{proof}

Hence Theorem \ref{thm:aff-tcat} is proved.

\begin{remark}\label{rem:hcx-for-tb}
Given an object $(M,K)$ of $\Geom(\C,T)$, recall that $T_*(M,K) = (TM,K_T)$ (\ref{thm:sm-aff-func}).  By \ref{thm:sm-acx} we obtain the following explicit formula for the connection $K_T:T^3M \rightarrow T^2M$ on $TM$:
$$K_T = c^{[-2]}_MT(K)c_M = (c_{TM}T(c_M))^{-1}T(K)c_M = T(c_M)c_{TM}T(K)c_M.$$
For brevity, we will often write this formula as 
	\[ K_T = T(c)cT(K)c. \]
	
Letting $H$ be the horizontal connection induced by $K$, we deduce by \ref{thm:sm-hcx} that the associated horizontal connection $H_T:T_2TM  \rightarrow T^3M$ induced by $K_T$ is
$$H_T = (c_M \times c_M)T(H)c^{[2]}_M = (c_M \times c_M)T(H)c_{TM}T(c_M)$$
where $c_M \times c_M:T_2TM = T^2M \times_{TM} T^2M \rightarrow T^2M \times_{TM} T^2M = TT_2M$ is induced by $c_M$ on each factor.  For brevity, we write this formula also as 
$$H_T = (c\times c)T(H)cT(c).$$
\end{remark}

The following result will be useful when working with the map $K_T$:

\begin{lemma}\label{lemmaKT}
If $(M,K) \in \Geom(\C,T)$, then:
\begin{enumerate}[(i)]
	\item $K_TT(p) = T^2(p)K$;
	\item $T(\ell)K_T = K \ell$.
\end{enumerate}
\end{lemma}
\begin{proof}
(i) asserts precisely that $p_M:T_*(M,K) = (TM,K_T) \rightarrow (M,K)$ is a morphism in $\Geom(\C,T)$, but this is immediate from \ref{thm:strmr-aff-conn}/\ref{thm:aff-conn-tcats} since we have a natural transformation $p_*:T_* \Rightarrow 1:\Geom(\C,T) \rightarrow \Geom(\C,T)$ with components ${p_*}_{(M,K)} = p_M$.  For (ii) we compute that
\begin{eqnarray*}
&   & T(\ell)K_T \\
& = & T(\ell)T(c)cT(K)c \\
& = & T(\ell c) c T(K)c \\
& = & T(\ell) c T(K) c \\
& = & K \ell c \mbox{ (by proposition \ref{prop:propK}.c)} \\
& = & K \ell
\end{eqnarray*}
as required.
\end{proof}


\subsection{Alternate characterizations of flat torsion-free connections} The fact that a connection on an object $M$ can be lifted to a connection on $TM$ (and then on $T^2M$, etc.) leads to two alternate characterizations of when a connection is flat and torsion-free.  One of these characterizations ((iii) in the result below) effectively says that ``a connection is flat torsion-free if and only if it is connection-preserving''.  If we think of a connection $K: T^2M \to TM$ as a `multiplication', the second characterization says that this operation is `associative'.  These characterizations appear to be new in standard differential geometry.

\begin{theorem}\label{propFTFConditions}
Suppose that $(M,K) \in \Geom(\C,T)$.  Then the following are equivalent:
\begin{enumerate}[(i)]
	\item $K$ is flat and torsion-free.
	\item $K_T K = T(K)K$.
	\item $K$ is a morphism in $\Geom(\C,T)$ from $(T^2M,K_{T^2})$ to $(TM,K_T)$.
\end{enumerate}
\end{theorem}
\begin{proof}
We first prove that (i) implies (ii).  Assuming that $K$ is flat and torsion-free, consider
\begin{eqnarray*}
&   & K_TK \\
& = & T(c)cT(K)cK \\
& = & T(c)cT(K)K \mbox{ (since $K$ torsion-free)} \\
& = & T(c)T(K)K \mbox{ (since $K$ flat)} \\
& = & T(cK)K \\
& = & T(K)K \mbox{ (since $K$ torsion-free).} 
\end{eqnarray*}
so that we have (ii).  \\

Next, we prove that (i) implies (iii).  Assuming that $K$ satisfies (i), we can apply Theorem \ref{thm:aff-conn-tcats} and Remark \ref{rem:hcx-for-tb} to deduce that $K_T$ also satisfies (i).  Hence, since we have already proved that (i) implies (ii), we deduce that both $K$ and $K_T$ satisfy (ii), a fact that we shall use in the following computations.  For $K$ to be a morphism in $\Geom(\C,T)$ between the objects in (iii), we must show that
	\[ K_{T^2}T(K) = T^2(K)K_T. \]
These are both maps into $T^2M$.  Now by Corollary \ref{thm:acx}, $T^2M$ is the fibre product of three copies of $TM$, with projections $K, T(p), p$.  So, to show the equality of the above maps, it suffices to show their equality when followed by these three projections.  For the equality with $K$, consider
\begin{eqnarray*}
&    & K_{T^2}T(K)K \\
& = & K_{T^2}K_T K \mbox{ (by (ii))} \\
& = & T(K_T)K_TK \mbox{ (by (ii), applied to $K_T$)} \\
& = & T(K_T)T(K)K \mbox{ (by (ii))} \\
& = & T(K_T K)K \\
& = & T(T(K)K)K \mbox{ (by (ii))} \\
& = & T^2(K)T(K)K \\
& = & T^2(K)K_T K \mbox{ (by (ii))}
 \end{eqnarray*}
For the equality with $p$, consider
\begin{eqnarray*}
&   & K_{T^2}T(K)p \\
& = & K_{T^2}pK \mbox{ (by naturality of $p$)} \\
& = & ppK \mbox{ (by proposition \ref{prop:propK}.b)} \\
& = & T^2(K)pp \mbox{ (by naturality of $p$)} \\
& = & T^2(K)K_Tp \mbox{ (by proposition \ref{prop:propK}.b)}
\end{eqnarray*}
Finally, for the equality with $T(p)$, consider
\begin{eqnarray*}
&   & K_{T^2}T(K)T(p) \\
& = & K_{T^2}T(Kp) \\
& = & K_{T^2}T(pp) \mbox{ (by proposition \ref{prop:propK}.b)} \\
& = & K_{T^2}T(p)T(p) \\
& = & T^2(p)K_TT(p) \mbox{ (by \ref{lemmaKT}(i))} \\
& = & T^2(p)T^2(p)K \mbox{ (by \ref{lemmaKT}(i))} \\
& = & T^2(pp)K \\
& = & T^2(Kp)K \mbox{ (by proposition \ref{prop:propK}.b)} \\
& = & T^2(K)T^2(p)K \\
& = & T^2(K)K_T T(p) \mbox{ (by lemma \ref{lemmaKT}.i)}
\end{eqnarray*}
as required. \\

We will now prove (iii) implies (ii).  Suppose that $T^2(K)K_T = K_{T^2}T(K)$.  Then composing both sides of the equation on the left by $T(\ell)$ and on the right by $K$, we get
\begin{eqnarray*}
T(\ell)T^2(K)K_T K & = & T(\ell)K_{T^2}T(K)K \\
T(\ell T(K))K_T K & = & K_T \ell T(K)K \mbox{ (using lemma \ref{lemmaKT}, applied to $K_T$)} \\
T(K \ell)K_T K & = & K_T K \ell K \mbox{ (by proposition \ref{prop:propK}.c)} \\
T(K)T(\ell)K_T K & = & K_T K \mbox{ (by proposition \ref{prop:propK}.a)} \\
T(K)K \ell K & = & K_T K \mbox{ (by lemma \ref{lemmaKT})} \\
T(K)K & = & K_T K \mbox{ (by proposition \ref{prop:propK}.a)}
\end{eqnarray*}
so that we have (ii). \\

Finally, we will show that (ii) implies (i).  Suppose that $K_T K = T(K)K$; in other words, 
	\[ T(c)cT(K)cK = T(K)K \ (\star). \]
Composing both sides of this equation on the left by $T(\ell)c$ gives
\begin{eqnarray*}
T(\ell)cT(c)cT(K)cK & = & T(\ell)cT(K)K \\
c \ell c T(K) c K & = & T(\ell)cT(K)K \mbox{ (by coherence of $\ell$ and $c$)} \\
c \ell T(K) c K & = & T(\ell)cT(K)K \\
c K \ell cK & = & K \ell K \mbox{ (by proposition \ref{prop:propK}.c)} \\
c K \ell K & = & K \mbox{ (by proposition \ref{prop:propK}.a)} \\
cK & = & K \mbox{ (by proposition \ref{prop:propK}.a)}
\end{eqnarray*}
So we have proven that $K$ is torsion-free.  Applying $cK = K$ to $\star$, we get
	\[ T(c)cT(K)K = T(K)K. \]
Now apply $T(c)$ to both sides of this equation to get
\begin{eqnarray*}
T(c)T(c) c T(K)K & = & T(c)T(K)K \\
c T(K)K & = & T(cK)K \\
c T(K)K & = & T(K)K \mbox{ (since $K$ torsion-free)}
\end{eqnarray*}
so that $K$ is flat.  Thus, we have proven (i).  

Altogether, we have proven 
	\[ (i) \Rightarrow (iii) \Rightarrow (ii) \Rightarrow (i), \]
and so all three conditions are equivalent.
\end{proof}

This result also allows us to prove several useful results about objects in the affine category.

\begin{corollary}\label{corTKK}
If $(M,K) \in \Aff(\C,T)$ then:
\begin{enumerate}[(i)]
	\item $K_TK = T(K)K$;
	\item $K$ is a morphism in $\Aff(\C,T)$ from $(T^2M,K_{T^2})$ to $(TM,K_T)$ and provides $(M,K)$ with the structure of a flat torsion-free connection in $\Aff(\C,T)$.   
	\item The maps from (ii) form the components of a natural transformation from $T^2_*$ to $T_*: \Aff(\C,T) \to \Aff(\C,T)$.   
\end{enumerate}
\end{corollary}
\begin{proof}
(i) and the fact that $K$ is a morphism in $\Aff(\C,T)$ were proved in the theorem.  Moreover, as the tangent structure on $\Aff(\C,T)$ is lifted from $(\C,T)$, this also shows that $K$ is a flat torsion-free connection on $(M,K)$ in the tangent category $\Aff(\C,T)$.  

Finally, that these maps form a natural transformation from $T^2_*$ to $T_*$ on $\Aff(\C,T)$ follows directly from the definition of maps in $\Aff(\C,T)$, namely that such maps preserve the associated connections of the objects.  
\end{proof}

\section{The 2-comonad of affine geometric spaces}

In \ref{thm:geom-2func-to-cat} we saw that there are 2-functors $\Geom,\Aff:\Tan \rightarrow \Cat$ that send each tangent category $(\C,T)$ to the categories  of geometric spaces and affine geometric spaces in $(\C,T)$, respectively.  As a consequence we found that $\Geom(\C,T)$ and $\Aff(\C,T)$ are tangent categories (\ref{thm:aff-conn-tcats}), so it is natural to wonder whether $\Geom$ and $\Aff$ lift to 2-functors valued in $\Tan$.  We now address this question, and we show that $\Aff$ underlies a 2-comonad, whose coalgebras are tangent categories whose objects carry affine geometric structure.  Here we employ the standard notion of (strict) 2-monad (as employed, for example, in \cite{BKP}). 

\begin{definition}
Let $(\C,T)$ be a tangent category.  Each of the following fibre products in $\C$ will be called a \textbf{basic fibre product} in $(\C,T)$: (1) Each fibre product of the form $T_nM$, and (2) each pullback witnessing the \textnormal{universality of the vertical lift} \cite[Def. 2.1]{CC3}.  A \textbf{class of endemic fibre products} in $(\C,T)$ is a class $\sF$ of finite fibre product diagrams in $\C$ that is closed under the application of $T$ and contains each basic fibre product.  There is clearly a smallest class of endemic fibre products in $(\C,T)$, consisting of the fibre product diagrams obtained by repeatedly applying $T$ to the basic fibre product diagrams. 
\end{definition}

Concretely, we shall represent finite fibre product diagrams in $\C$ as certain functors $D:\J_n \rightarrow \C$ on categories $\J_n$ defined as follows.  For each natural number $n$, $\J_n$ is a partially ordered set with $n+2$ distinct elements $1,2,...,n,\bot,\top$, in which $\bot$ is a bottom element, $\top$ is a top element, and the remaining elements $1,2,...,n$ are mutually incomparable.

\begin{definition}
There is a 2-category $\Tan_e$ whose objects $(\C,T,\sF)$ are tangent categories with a given class of endemic fibre products $\sF$.  A 1-cell $F:(\C,T,\sF) \rightarrow (\C',T',\sF')$ in $\Tan_e$ is a strong morphism of tangent categories that preserves endemic fibre products, i.e. sends fibre product diagrams in $\sF$ to fibre product diagrams in $\sF'$.  The 2-cells in $\Tan_e$ are simply tangent transformations.
\end{definition}

\begin{proposition}\label{thm:geomct-endfps}
Let $(\C,T,\sF)$ be a tangent category with endemic fibre products, and let $U = U_\C:\Geom(\C,T) \rightarrow \C$ denote the forgetful functor.  Then $\Geom(\C,T)$ carries a class of endemic fibre products $U^*(\sF)$, consisting of all diagrams of the form $D:\J_n \rightarrow \Geom(\C,T)$ with $UD \in \sF$.  The functor $U$ underlies a strict morphism of tangent categories, which in turn underlies a 1-cell $U:(\Geom(\C,T),T_*,U^*(\sF)) \rightarrow (\C,T,\sF)$ in $\Tan_e$.  Further, the analogous results hold with each of $\Aff$, $\GeomFlat$, and $\GeomTf$ in place of $\Geom$.
\end{proposition}
\begin{proof}
Given any diagram $D:\J_n \rightarrow \Geom(\C,T)$ in $U^*(\sF)$, we know that $UD \in \sF$ and hence $T^kUD \in \sF$ for every $k \in \N$, by induction on $k$.  Therefore $UD$ is a fibre product diagram in $\C$ that is preserved by each $T^k$, so by \ref{thm:refl-lim-cx} we deduce that $D$ is a fibre product diagram in $\Geom(\C,T)$.  Note also that $T_*D \in U^*(\sF)$, since $UT_*D = TUD \in \sF$.

Hence $U^*(\sF)$ is a class of fibre product diagrams in $\Geom(\C,T)$, and $U^*(\sF)$ is closed under the application of $T_*$.  In view of the construction of the tangent structure on $\Geom(\C,T)$ in \ref{thm:aff-conn-tcats} and \ref{thm:fps-in-aff}, it is clear that the basic fibre products in $\Geom(\C,T)$ are sent by $U$ to basic fibre products in $(\C,T)$ and hence lie in $U^*(\sF)$.
\end{proof}

We shall now prove that $\Geom$ lifts to a 2-endofunctor on $\Tan_e$.  We begin with the following general observation:

\begin{lemma}\label{lem:alpha-ttransf}
Let $(F,\alpha):(\C,T) \rightarrow (\C',T')$ be a morphism of tangent categories.  Then $\alpha:FT \Rightarrow T'F$ underlies a tangent transformation
$$\alpha\;:\;(F,\alpha) \circ (T,c) \Longrightarrow (T',c') \circ (F,\alpha)\;:\;(\C,T) \rightarrow (\C',T')$$
where $c,c'$ denote the canonical flips carried by $\C,\C'$, respectively (cf. \ref{thm:tstr-sm-tt}).
\end{lemma}
\begin{proof}
Using the definition of composition of morphisms of tangent categories (\ref{def:str-mor}), we first note that $(F,\alpha) \circ (T,c) = (FT,c * \alpha)$ and $(T',c') \circ (F,\alpha) = (T'F, \alpha * c')$ where $c * \alpha = F(c)\alpha_T:FTT \Rightarrow T'FT$ and $\alpha * c' = T'(\alpha)c'_F:T'FT \Rightarrow T'T'F$.  Hence it suffices to show that the diagram
$$
\xymatrix{
FTT \ar[d]_{\alpha_T} \ar[r]^{c * \alpha} & T'FT \ar[d]^{T'(\alpha)}\\
T'FT \ar[r]_{\alpha * c'} & T'T'F
}
$$
commutes.  Indeed,
$$\alpha_T(\alpha * c') = \alpha_TT'(\alpha)c'_F = \alpha^{[2]}c'_F = F(c)\alpha^{[2]} = F(c)\alpha_TT'(\alpha) = (c * \alpha)T'(\alpha)$$
since $(F,\alpha)$ is a  morphism of tangent categories (\ref{def:str-mor}).
\end{proof}

\begin{theorem}
There is a 2-functor
$$\Geom\;:\;\Tan_e \rightarrow \Tan_e$$
sending each tangent category with endemic fibre products $(\C,T,\sF)$ to the tangent category $\Geom(\C,T)$ of geometric spaces in $(\C,T)$, equipped with its associated class of endemic fibre products $U^*(\sF)$ (\ref{thm:geomct-endfps}).  Similarly, there are 2-functors $\GeomFlat,\GeomTf,\Aff:\Tan_e \rightarrow \Tan_e$ sending $(\C,T,\sF)$ to the tangent categories of flat, torsion-free, and affine geometric spaces in $(\C,T)$, respectively.
\end{theorem}
\begin{proof}
We shall treat the case of $\Geom$; the other 2-functors are obtained similarly, using \ref{thm:geom-2func-to-cat}.  Letting $F:(\C,T,\sF) \rightarrow (\C',T',\sF')$ be a 1-cell in $\Tan_e$, we know that the associated isomorphism $\alpha^F:FT \Rightarrow T'F$ is a tangent transformation and so is a 1-cell in $\Tan$.  Hence we can apply $\Geom:\Tan \rightarrow \Cat$ to $\alpha^F$ in order to obtain an invertible 2-cell $\alpha^F_*:F_*T_* \Rightarrow T'_*F_*:\Geom(\C,T) \rightarrow \Geom(\C',T')$ in $\Cat$.  We claim that $(F_*,\alpha^F_*):\Geom(\C,T) \longrightarrow \Geom(\C',T')$ is a 1-cell in $\Tan_e$.  Indeed, employing the notation of \ref{thm:geomct-endfps}, we reason that for each $D \in U_\C^*(\sF)$ the composite $F_*D$ lies in $U_{\C'}^*(\sF')$, since $U_{\C'}F_*D = FU_\C D \in \sF'$ because $U_\C D \in \sF$.  Hence $F_*$ preserves endemic fibre products and so, in particular, sends basic fibre product diagrams to fibre product diagrams.  Also, since $\Geom:\Tan \rightarrow \Cat$ is 2-functorial and $F$ is a strong morphism of tangent categories, it follows that $(F_*,\alpha^F_*)$ satisfies the equational axioms for a morphism of tangent categories (\ref{def:str-mor}).

This defines the needed assignment on 1-cells, and the functoriality of this assignment readily follows from the 2-functoriality of $\Geom:\Tan \rightarrow \Cat$.  Given a 2-cell $\phi:F \Rightarrow G:(\C,T,\sF) \rightarrow (\C',T',\sF')$ in $\Tan_e$, we can apply $\Geom:\Tan \rightarrow \Cat$ to obtain a natural transformation $\phi_*:F_* \Rightarrow G_*:\Geom(\C,T) \rightarrow \Geom(\C',T')$ in $\Cat$, which is in fact a tangent transformation
$$\phi_*:(F_*,\alpha^F_*) \Rightarrow (G_*,\alpha^G_*):\Geom(\C,T) \rightarrow \Geom(\C',T')$$
since ${\phi_*}_{T_*}\alpha^G_* = (\phi_T\alpha^G)_* = (\alpha^FT'(\phi))_* = \alpha^F_*T'_*(\phi_*)$ by the 2-functoriality of $\Geom:\Tan \rightarrow \Cat$.  Again using the latter 2-functoriality, the result now follows.
 \end{proof}

\begin{theorem}\label{thm:2-comonad-aff}
There is a 2-comonad $\AAff = (\Aff,\varepsilon,\delta)$ on $\Tan_e$ whose underlying 2-functor
$$\Aff:\Tan_e \rightarrow \Tan_e$$
sends each tangent category with endemic fibre products, $(\C,T,\sF)$, to the tangent category $\Aff(\C,T)$ of affine geometric spaces in $(\C,T)$.  The counit 1-cell
$$\varepsilon_{(\C,T,\sF)}:\Aff(\C,T) \rightarrow (\C,T)$$
in $\Tan_e$ is the forgetful functor, and the comultiplication 1-cell
\begin{equation}\label{eq:delta}\delta_{(\C,T,\sF)}:\Aff(\C,T) \rightarrow \Aff(\Aff(\C,T))\end{equation}
sends each affine geometric space $(M,K)$ in $(\C,T)$ to the affine geometric space $((M,K),K)$ in $\Aff(\C,T)$.
\end{theorem}
\begin{proof}
By \ref{thm:geomct-endfps}, we know that each forgetful functor $\varepsilon_{(\C,T,\sF)}$ is a strict morphism of tangent categories and is also a 1-cell in $\Tan_e$.  Further, it is immediate from the definitions that this defines a 2-natural transformation $\varepsilon:\Aff \rightarrow 1_{\Tan_e}$.

With regard to the comultiplication $\delta$, recall that if $(M,K)$ is an affine geometric space in a tangent category $(\C,T)$, then $K:T^2_*(M,K) \rightarrow T_*(M,K)$ is a flat torsion-free connection on $(M,K)$ in $\Aff(\C,T)$ (\ref{corTKK}), so $((M,K),K)$ is an affine geometric space in $\Aff(\C,T)$, i.e. an object of $\Aff(\Aff(\C,T))$.  For each object $(\C,T,\sF)$ of $\Tan_e$, this defines $\delta_{(\C,T,\sF)}$ on objects.  Given a morphism $f:(M,K) \rightarrow (M',K')$ in $\Aff(\C,T)$, it is immediate that $\delta_{(\C,T,\sF)}(f) = f:((M,K),K) \rightarrow ((M',K'),K')$ defines a morphism in $\Aff(\Aff(\C,T))$.  Thus we obtain a functor $\delta_{(\C,T,\sF)}$ as in \eqref{eq:delta}.  The diagram of functors
\begin{equation}\label{eq:counit1}
\xymatrix{
\Aff(\C,T) \ar@{=}[dr] \ar[r]^(.45){\delta_{(\C,T,\sF)}} & \Aff(\Aff(\C,T)) \ar[d]^{\varepsilon_{\Aff(\C,T)}}\\
& \Aff(\C,T)
}
\end{equation}
clearly commutes, and $\varepsilon_{\Aff(\C,T)}$ is a strict morphism of tangent categories and is also a faithful functor, so it follows that $\delta_{(\C,T,\sF)}$ is a strict morphism of tangent categories.  Using the commutativity of this diagram, we also find that $\delta_{(\C,T,\sF)}$ preserves endemic fibre products (since $\varepsilon_{\Aff(\C,T)}$ \textit{reflects} endemic fibre products).  Hence $\delta_{(\C,T,\sF)}$ is a 1-cell in $\Tan_e$.  

This defines a natural transformation $\delta:\Aff \rightarrow \Aff \circ \Aff$, since if $F:(\C,T,\sF) \rightarrow (\C',T',\sF')$ is a 1-cell in $\Tan_e$ then the diagram
$$
\xymatrix{
\Aff(\C,T) \ar[d]_{F_*} \ar[r]^(.45)\delta & \Aff(\Aff(\C,T)) \ar[d]^{(F_*)_*}\\
\Aff(\C',T') \ar[r]_(.45)\delta             & \Aff(\Aff(\C',T'))
}
$$
commutes.  Indeed, for each object $(M,K)$ of $\Aff(\C,T)$ we compute that
$$
\begin{array}{cccccccc}
(F_*)_*(\delta(M,K)) & = & (F_*)_*((M,K),K)   & = & (F_*(M,K),K_{F_*}) & = & ((FM,K_F),K_{F_*})\\
                    & = & ((FM,K_F),K_F)     & = & \delta(FM,K_F) & = & \delta(F_*(M,K)) & 
\end{array}
$$
since the definitions of $K_F$ and $K_{F_*}$ readily entail that
$$K_{F_*} = K_F:{T'_*}^2(FM,K_F) \rightarrow T'_*(FM,K_F)$$
by the 2-functoriality of $\Aff = (-)_*:\Tan \rightarrow \Cat$.  Commutativity on morphisms is immediate.  The resulting natural tranformation $\delta$ is, moreover, 2-natural, as one readily verifies.

We already know that $(\Aff,\varepsilon,\delta)$ satisfies one of the co-unit laws \eqref{eq:counit1}.  For the other, we must show that the composite
$$\Aff(\C,T) \xrightarrow{\delta_{(\C,T,\sF)}} \Aff(\Aff(\C,T)) \xrightarrow{U_*} \Aff(\C,T)$$
is the identity, where $U = \varepsilon_{(\C,T,\sF)}$.  But this is nearly immediate, since this composite sends each object $(M,K)$ to $(M,K_U)$, while $K_U = K$ since the forgetful functor $U$ is a strict morphism of tangent categories.

For the co-associativity law, we must show that the diagram
$$
\xymatrix{
\Aff(\C,T) \ar[d]_{\delta_{(\C,T,\sF)}} \ar[r]^(.45){\delta_{(\C,T,\sF)}} & \Aff(\Aff(\C,T)) \ar[d]^{(\delta_{(\C,T,\sF)})_*}\\
\Aff(\Aff(\C,T)) \ar[r]_(.45){\delta_{\Aff(\C,T)}} & \Aff(\Aff(\Aff(\C,T)))
}
$$
commutes.  But the definitions readily entail that both composites send an object $(M,K)$ of $\Aff(\C,T)$ to $(((M,K),K),K)$, and the commutativity on arrows is immediate.
\end{proof}

\begin{definition}
An \textbf{affine tangent category} is an Eilenberg-Moore $\AAff$-coalgebra, for the 2-comonad $\AAff = (\Aff,\varepsilon,\delta)$ defined in \ref{thm:2-comonad-aff}.  Hence affine tangent categories are the objects of a 2-category, namely the Eilenberg-Moore 2-category for the 2-comonad $\AAff$.
\end{definition}

\begin{proposition}
An affine tangent category is equivalently given by a tangent category with endemic fibre products $(\C,T,\sF)$ in which each object $M$ is equipped with a flat torsion-free connection $K_M$ such that
\begin{enumerate}
\item every morphism $f:M \rightarrow N$ in $\C$ preserves the given connections $K_M$ and $K_N$, in the sense that $K_MT(f) = T^2(f)K_N$, and 
\item for each object $M$ of $\C$, the following diagram commutes:
$$
\xymatrix{
T^3M \ar[d]_{K_{TM}} \ar[r]^{T(c_M)} & T^3M \ar[r]^{c_{TM}} & T^3M \ar[d]^{T(K_M)}\\
T^2M \ar[rr]_{c_M} & & T^2M
}
$$
\end{enumerate} 
\end{proposition}
\begin{proof}
Suppose that we are given an $\AAff$-coalgebra $((\C,T,\sF),A)$, i.e. an object $(\C,T,\sF)$ of $\Tan_e$ together with a 1-cell $A:(\C,T,\sF) \rightarrow \Aff(\C,T)$ in $\Tan_e$ making the following diagrams commute in $\Tan$:
\begin{equation}\label{eq:em}
\xymatrix{
(\C,T) \ar@{=}[dr] \ar[r]^(.45){A} & \Aff(\C,T) \ar[d]^{\varepsilon_{(\C,T,\sF)}} & (\C,T) \ar[d]_A \ar[r]^A & \Aff(\C,T) \ar[d]^{A_*}\\
                  & (\C,T) & \Aff(\C,T) \ar[r]_(.45){\delta_{(\C,T,\sF)}} & \Aff(\Aff(\C,T))
}
\end{equation}
Since $U = \varepsilon_{(\C,T,\sF)}$ is the forgetful functor, the unit law $UA = 1$ entails that $A$ must send each object $M$ of $\C$ to an object of the form $(M,K_M)$ with $K_M$ a flat torsion-free connection on $M$.  For each morphism $f$ in $\C$, as in 1, we have $U(A(f)) = f$, so that $A(f) = f:(M,K_M) \rightarrow (N,K_N)$ and 1 holds.

We claim that $A$ is necessarily a \textit{strict} morphism of tangent categories.  Indeed, since $UA = 1_{(\C,T)}$ in $\Tan$ and $U$ is a strict morphism, it follows that the structural isomorphism $\alpha^A:AT \rightarrow T^*A$ has $U(\alpha^A) = 1_T:UAT = T \rightarrow T = UT^*A$, so that each of its components
$$\alpha^A_M:ATM = (TM,K_{TM}) \rightarrow T^*AM = (TM,(K_M)_T)$$
is a morphism in $\Aff(\C,T)$ whose underlying morphism in $\C$ is $1_{TM}$.  Hence $1_{TM}$ preserves the connections $K_{TM}$ and $(K_M)_T$, in the sense that $K_{TM}T(1_{TM}) = T^2(1_{TM})(K_M)_T$, i.e. $K_{TM} = (K_M)_T$, so 2 holds since by definition $(K_M)_T = T(c_M)c_{TM}T(K_M)c_M$ and $c_M = c_M^{-1}$.  We now deduce also that $AT = T^*A$ as functors and that $\alpha^A$ is the identity transformation on $AT$.

Conversely, suppose that $(\C,T,\sF)$ is a tangent category with endemic fibre products and an assignment $M \mapsto K_M$ satisfying 1 and 2.  Then we can define a functor $A:\C \rightarrow \Aff(\C,T)$ on objects by $AM = (M,K_M)$ and on arrows by $A(f) = f$, whereupon $UA = 1_\C$ as functors.  Hence it is immediate that $A$ preserves endemic fibre products, since $U$ reflects endemic fibre products.  In view of the above, 2 asserts precisely that $K_{TM} = (K_M)_T$ for each object $M$, i.e. that $ATM = T^*AM$ as objects.  But it then follows immediately that $AT = T^*A$ as functors, so since $UA = 1_\C$ and $U$ is faithful and is a strict morphism of tangent categories, it follows that $A$ is a strict morphism of tangent categories.  Hence $A$ is a 1-cell in $\Tan_e$, and clearly $UA = 1_{(\C,T,\sF)}$ in $\Tan_e$.  In the rightmost diagram in \eqref{eq:em}, each 1-cell is a strict morphism of tangent categories, so the diagram commutes in $\Tan_e$ as soon as the underlying diagram in $\Cat$ commutes.  Indeed, for each object $M$ of $\C$ we compute that
$$A_*(AM) = A_*(M,K_M) = (AM,(K_M)_A) = ((M,K_M),K_M) = \delta_{(\C,T,\sF)}(AM)$$
since it follows readily from the definitions that $(K_M)_A = K_M:T^2_*(M,K_M) \rightarrow T_*(M,K_M)$, and the commutativity on arrows is immediate. 
\end{proof}

\begin{example}
Given any tangent category with endemic fibre products, $(\C,T,\sF)$, the category of affine geometric spaces $\Aff(\C,T)$ is an affine tangent category, since it is a cofree $\AAff$-coalgebra.
\end{example}

\section{Conclusion and future work}\label{sec:jub}

As noted above, we were largely prompted to investigate affine manifolds by the work of 
Beno\^it Jubin \cite{Jub}. As an indication to future work, we will summarize certain of his results, some of which we will generalize to the setting of tangent categories in a subsequent paper.  Beginning with the tangent functor on the category of smooth manifolds, Jubin proved the following:
\begin{theorem}[Jubin \cite{Jub}]
\emptybox
\begin{itemize}
\item \add{The tangent functor on the category of smooth manifolds carries a unique monad structure}. Using local coordinates, \add{the multiplication} $\mu\colon T^2M\rarr TM$ \add{is given }by

\[\mu\colon T^2M\rarr TM\colon (x,v,\dot{x},\dot{v})\mapsto (x,v+\dot{x})\add{\;.}\]

\noindent The unit $\eta\colon id\rarr T$ is given by the zero section.
\item There are no comonad structures on the tangent functor on the category of smooth manifolds.
\end{itemize}
\end{theorem}

\add{This monad structure exists in} any tangent category (see \cite[Proposition 3.4]{CC1}).   But 
uniqueness depends crucially on the setting of smooth manifolds as does the lack of comonads. These results 
do not follow from the axioms 
of tangent category (in particular, Jubin's later results show that the tangent category of affine manifolds has multiple monad and comonad structures on its tangent functor).  

\smallskip

Jubin also studies the category of affine manifolds, $\Aff$, whose morphisms are locally affine maps \add{(see \S}\ref{Aff}\add{)}, and shows that the tangent functor on smooth manifolds lifts to an endofunctor on $\Aff$ (as discussed in \S \ref{Aff}).

Working in the smaller category \Aff, one has significantly more freedom to define structures on the tangent
functor. Indeed Jubin completely characterizes \add{all} monad and comonad structures on the tangent functor on \Aff.
See Proposition 3.2.1 of \cite{Jub}. Since \add{this takes place} in a category of smooth manifolds, \add{one} can
define the necessary structural maps using local coordinates. 

\begin{theorem} [Jubin]
\emptybox
\begin{itemize}\label{a-b}
\item The only monad structures on the tangent functor on the category \Aff\ are indexed by the real numbers\add{,} and for a fixed real number $a$ \add{the monad multiplication }is given by the following:

\[\mu^a\colon T^2M\rarr TM\colon (x,v,w,d)\mapsto (x,v+w+ad)\]

\noindent In this case, the unit map for the monad must be the zero section. 
\item The only comonad structures on the tangent functor on the category \Aff\ are indexed by the real numbers\add{,} 
and for a fixed real number $b$ \add{the comultiplication }is given by the following:

\[\delta^b\colon TM\rarr T^2M\colon (x,v)\mapsto (x,v,v,bv)\]

\noindent In this case, the counit map for the comonad is given by the bundle projection.
\end{itemize}
\end{theorem}

Furthermore, Jubin discovered that \add{the above} monad and comonad structures interact to form bimonads 
\cite{MW}. These are endofunctors with monad and comonad structure as well as a {\it mixed 
distributive law} that relates the two structures and satisfies some additional axioms.   

\begin{theorem}[Jubin]\label{thm:jub_bimnd} The tangent functor $T$ equipped with its $a$-monad structure  
and $b$-comonad structure (Theorem \ref{a-b}) is a bimonad. (For relevant categorical definitions, see \cite{MW}).
The formula for the mixed distribution is given by (using local coordinates):
 
\[\lambda^{a,b}\colon T^2M\rarr T^2M\colon(x,v,w,d) \mapsto(x,w,v+w + ad,bw-d)\]
\end{theorem}

In a sequel, we will present an abstraction of certain of these results to the tangent category setting.
The consequences of the existence of these structures at this abstract level remain to be studied.
	
\bibliographystyle{plain}


\end{document}